\documentclass[english,american]{amsart}
\usepackage[T1]{fontenc}
\usepackage[latin9]{inputenc}
\usepackage{amsthm}
\usepackage{amssymb}
\usepackage{esint}
\usepackage{graphicx}
\numberwithin{equation}{section} 
\numberwithin{figure}{section} 
\theoremstyle{plain}
 \theoremstyle{definition}
 \newtheorem*{defn*}{Definition}
\newtheorem{thm}{Theorem}[section]
  \theoremstyle{plain}
  \newtheorem{cor}[thm]{Corollary}
  \theoremstyle{plain}
  \newtheorem{lem}[thm]{Lemma}
  \theoremstyle{definition}
  \newtheorem{defn}[thm]{Definition}
 \theoremstyle{definition}
  \newtheorem{example}[thm]{Example}
  \theoremstyle{remark}
  \newtheorem*{rem*}{Remark}
  \theoremstyle{plain}
  \newtheorem{prop}[thm]{Proposition}

\def\Xint#1{\mathchoice
{\XXint\displaystyle\textstyle{#1}}%
{\XXint\textstyle\scriptstyle{#1}}%
{\XXint\scriptstyle\scriptscriptstyle{#1}}%
{\XXint\scriptscriptstyle\scriptscriptstyle{#1}}%
\!\int}
\def\XXint#1#2#3{{\setbox0=\hbox{$#1{#2#3}{\int}$}
\vcenter{\hbox{$#2#3$}}\kern-.5\wd0}}

\def\dashint{\Xint-}

\usepackage{babel}

\begin{document}

\title{Hankel determinants of Dirichlet series}

\author{Hartmut Monien}
\begin{abstract}
We derive a general expression for the Hankel determinants of a Dirichlet
series $F(s)$ and derive the asymptotic behavior for the special
case that $F(s)$ is the Riemann zeta function. In this case the Hankel
determinant is a discrete analogue of the Selberg integral
and can be viewed as a matrix integral with discrete measure. We briefly
comment on its relation to Plancherel measures.
\end{abstract}

\address{Bethe Center for Theoretical Physics, University Bonn, Germany}

\maketitle

\section{Introduction}

In this paper we will consider the Hankel determinants\[
H_{1}^{(0)}[\zeta]=\zeta(2),\; H_{2}^{(0)}[\zeta]=\left|\begin{array}{cc}
\zeta(2) & \zeta(3)\\
\zeta(3) & \zeta(4)\end{array}\right|,\; H_{3}^{(0)}[\zeta]=\left|\begin{array}{ccc}
\zeta(2) & \zeta(3) & \zeta(4)\\
\zeta(3) & \zeta(4) & \zeta(5)\\
\zeta(4) & \zeta(5) & \zeta(6)\end{array}\right|\ldots\]
and\[
H_{1}^{(1)}[\zeta]=\zeta(3),\; H_{2}^{(1)}[\zeta]=\left|\begin{array}{cc}
\zeta(3) & \zeta(4)\\
\zeta(4) & \zeta(5)\end{array}\right|,\; H_{3}^{(1)}[\zeta]=\left|\begin{array}{ccc}
\zeta(3) & \zeta(4) & \zeta(5)\\
\zeta(4) & \zeta(5) & \zeta(6)\\
\zeta(5) & \zeta(6) & \zeta(7)\end{array}\right|\ldots\]
and various generalizations of them. These determinants go very rapidly
to zero as the dimension of the matrix becomes large e.g.\begin{eqnarray*}
H_{100}^{(0)}[\zeta] & \approx & 4.9\times10^{-16684}\\
H_{100}^{(1)}[\zeta] & \approx & 4.3\times10^{-16871}\end{eqnarray*}
The ratios have (experimentally) a surprisingly simple asymptotic
expansion:\begin{eqnarray*}
-\frac{H_{n-1}^{(0)}[\zeta]H_{n}^{(1)}[\zeta]}{H_{n}^{(0)}[\zeta]H_{n+1}^{(1)}[\zeta]} & = & -\frac{1}{2n+1}+\frac{2}{(2n+1)^{2}}-\frac{7}{3}\frac{1}{(2n+1)^{3}}+\frac{16}{5}\frac{1}{(2n+1)^{4}}-\frac{41}{9}\frac{1}{(2n+1)^{5}}+\ldots\\
-\frac{H_{n+1}^{(0)}[\zeta]H_{n-1}^{(1)}[\zeta]}{H_{n}^{(0)}[\zeta]H_{n}^{(1)}[\zeta]} & = & -\frac{1}{2n}-\frac{1}{(2n)^{2}}+\frac{2}{3}\frac{1}{(2n)^{3}}-\frac{6}{5}\frac{1}{(2n)^{4}}+\frac{56}{45}\frac{1}{(2n)^{5}}+\ldots.\end{eqnarray*}
This recursion can be solved to yield explicit expressions for $H^{(r)}_n[\zeta]$. In fact more detailed numerical experiments by D. Zagier confirmed our findings
with the result\[
H_{n}^{(0)}[\zeta]=A^{(0)}\left(\frac{2n+1}{e\sqrt{e}}\right)^{-(n+\frac{1}{2})^{2}}\left(1+\frac{1}{24}\frac{1}{(2n+1)^{2}}-\frac{12319}{259200}\frac{1}{(2n+1)^{4}}+\frac{504407873}{217728}\frac{1}{(2n+1)^{6}}\ldots\right)\]
and\[
H_{n-1}^{(1)}[\zeta]=A^{(1)}\left(\frac{2n}{e\sqrt{e}}\right)^{-n^{2}+\frac{3}{4}}\left(1-\frac{17}{240}\frac{1}{(2n)^{2}}-\frac{199873}{7257600}\frac{1}{(2n)^{4}}-\frac{90789413}{1741824000}\frac{1}{(2n)^{6}}\ldots\right)\]
with the interesting observation that\begin{eqnarray*}
A^{(0)} & \approx & 0.351466738331\\
A^{(1)} & = & \frac{e^{9/8}}{\sqrt{6}}A^{(0)}.\end{eqnarray*}

We do not know how to prove the full asymptotic expansions but will
describe a method that lets one at least understand the weaker asymptotic
form\[
\log H_{n}^{(0)}[\zeta]\sim\log H_{n}^{(1)}[\zeta]\sim-n^{2}(\log(2n)-\frac{3}{2}).\]
We will also discuss interesting relations with a continuous version
(Selberg integral) and with the Plancherel measure of the symmetric
group. We start with the presentation of some results on the Hankel
determinants of Dirichlet series.

\section{Hankel determinants of Dirichlet series}
\begin{defn*}
Let $F(s)$ be a Dirichlet series with coefficients $f(n)$ i.e.\[
F(s)=\sum_{n=1}^{\infty}\frac{f(n)}{n^{s}}.\]
For $n$ integer, $n>0$, $r$ integer, $r\ge0$ we define the Hankel
determinant $H_{n}^{(r)}[F]$: \[
H_{n}^{(r)}[F]=\det\left(F(i+j+r)\right)_{1\le i,j\le n}\]
We will also use the notation $H_{n}[F]=H_{n}^{(0)}[F]$. Our first
result is:\end{defn*}
\begin{thm}
$H_{n}^{(r)}\left[F\right]$ is given by\[
H_{n}^{(r)}[F]=\frac{1}{n!}\sum_{m_{1},m_{2},\ldots m_{n}=1}^{\infty}\prod_{i=1}^{n}\frac{f(m_{i})}{m_{i}^{2n+r}}\prod_{i<j}(m_{i}-m_{j})^{2}.\]
\label{thm:HankelDeterminant}\end{thm}
\begin{proof}
We prove the case $r=0$ first. Using the definition of the determinant

\begin{eqnarray*}
H_{n}^{(0)}[F] & = & \sum_{\pi\in\mathfrak{S_{n}}}(-1)^{\pi}F(1+\pi(1))F(2+\pi(2))\cdots F(n+\pi(n))\\
 & = & \sum_{m_{1}\ldots m_{n}\ge1}\sum_{\pi\in\mathfrak{S_{n}}}\left(-1\right)^{\pi}\frac{f(m_{1})f(m_{2})\cdots f(m_{n})}{m_{1}^{1+\pi(1)}m_{2}^{2+\pi(2)}\ldots m_{n}^{n+\pi(n)}}\\
 & = & \sum_{m_{1}\ldots m_{n}\ge1}\frac{f(m_{1})f(m_{2})\cdots f(m_{n})}{m_{1}^{1}m_{2}^{2}\ldots m_{n}^{n}}\sum_{\pi\in\mathfrak{S_{n}}}\frac{\left(-1\right)^{\pi}}{m_{1}^{\pi(1)}m_{2}^{\pi(2)}\ldots m_{n}^{\pi(n)}}\\
 & = & \sum_{m_{1}\ldots m_{n}\ge1}\frac{f(m_{1})f(m_{2})\cdots f(m_{n})}{m_{1}^{2}m_{2}^{3}\ldots m_{n}^{n+1}}\prod_{i<j}\left(\frac{1}{m_{i}}-\frac{1}{m_{j}}\right)\end{eqnarray*}
where $\mathfrak{S_{n}}$ is the symmetric group and $(-1)^{\pi}$
is the sign of the permutation $\pi$. In the last line we have used
the Vandermonde determinant\[
\sum_{\pi\in\mathfrak{S_{n}}}\frac{\left(-1\right)^{\pi}}{m_{1}^{\pi(1)}m_{2}^{\pi(2)}\ldots m_{n}^{\pi(n)}}=\frac{1}{m_{1}m_{2}\cdots m_{n}}\prod_{i<j}\left(\frac{1}{m_{i}}-\frac{1}{m_{j}}\right).\]
Interchanging two summation variables say $m_{i}$ and $m_{j}$ with
$i\ne j$ the Vandermonde determinant changes sign. Summing over all
permutations of $\left\{ 1,2,\ldots n\right\} $ and dividing by the
number of permutations we obtain: \begin{eqnarray*}
H_{n}^{(0)}[F] & = & \frac{1}{n!}\sum_{m_{1}\ldots m_{n}\ge1}\frac{f(m_{1})f(m_{2})\cdots f(m_{n})}{m_{1}m_{2}\cdots m_{n}}\sum_{\pi\in\mathfrak{S_{n}}}\frac{\left(-1\right)^{\pi}}{m_{\pi(1)}^{1}m_{\pi(2)}^{2}\ldots m_{\pi(n)}^{n}}\prod_{i<j}\left(\frac{1}{m_{i}}-\frac{1}{m_{j}}\right)\\
 & = & \frac{1}{n!}\sum_{m_{1}\ldots m_{n}\ge1}\frac{f(m_{1})f(m_{2})\cdots f(m_{n})}{m_{1}m_{2}\cdots m_{n}}\sum_{\pi\in\mathfrak{S_{n}}}\frac{\left(-1\right)^{\pi}}{m_{1}^{\pi(1)}m_{2}^{\pi(2)}\ldots m_{n}^{\pi(n)}}\prod_{i<j}\left(\frac{1}{m_{i}}-\frac{1}{m_{j}}\right)\\
 & = & \frac{1}{n!}\sum_{m_{1}\ldots m_{n}\ge1}\frac{f(m_{1})f(m_{2})\cdots f(m_{n})}{m_{1}^{2}m_{2}^{2}\cdots m_{n}^{2}}\prod_{i<j}\left(\frac{1}{m_{i}}-\frac{1}{m_{j}}\right)^{2}\\
 & = & \frac{1}{n!}\sum_{m_{1}\ldots m_{n}\ge1}\frac{f(m_{1})f(m_{2})\cdots f(m_{n})}{(m_{1}m_{2}\cdots m_{n})^{2n}}\prod_{i<j}\left(m_{i}-m_{j}\right)^{2}\end{eqnarray*}
For $r>0$ we replace $f(n)\rightarrow f(n)n^{-r}$ which gives \[
H_{n}^{(r)}[F]=\frac{1}{n!}\sum_{m_{1}\ldots m_{n}\ge1}\frac{f(m_{1})f(m_{2})\cdots f(m_{n})}{(m_{1}m_{2}\cdots m_{n})^{2n+r}}\prod_{i<j}\left(m_{i}-m_{j}\right)^{2}\]
which completes the proof.\end{proof}
\begin{cor}\label{cor:multiplicative function}
If $f(n)$ is multiplicative 
it follows trivially from Theorem \ref{thm:HankelDeterminant}\begin{eqnarray*}
H_{n}^{(r)}[\zeta] & = & \frac{1}{n!}\sum_{m_{1},m_{2},\ldots m_{n}=1}^{\infty}\frac{f\left(m_{1}m_{2}\cdots m_{n}\right)}{(m_{1}m_{2}\ldots m_{n})^{2n+r}}\prod_{i<j}(m_{i}-m_{j})^{2}.\\
 & = & \sum_{m=1}^{\infty}\frac{f(m)}{m^{2n+r}}\left(\frac{1}{n!}\sum_{m_{1}\cdot m_{2}\cdots m_{n}=m}\prod_{i<j}\left(m_{i}-m_{j}\right)^{2}\right)\end{eqnarray*}
which is again a Dirichlet series. Note that\[
\sum_{m_{1}\cdot m_{2}\cdots m_{n}=m}\prod_{i<j}\left(m_{i}-m_{j}\right)^{2}\ge0\]
and vanishes if the number of prime factors of $m$ is less than $n-1$
and is divisible by $\left(\prod_{i=1}^{n}(i-1)!\right)^{2}$ since
each summand $\prod_{i<j}\left(m_{i}-m_{j}\right)^{2}$ is divisible
by $\left(\prod_{i=1}^{n}(i-1)!\right)$. For a proof see \cite{Chapman96}. \end{cor}
\begin{lem}
If $f(n)>0$ for all integers $n>0$ then\[
H_{n}^{(r)}[F]>0.\]
\end{lem}
\begin{proof}
The smallest $m$ contributing in the sum is $m=1\cdot2\cdots n=n!$.
For $m_{i}=\{1,2,\ldots n\}$ we have \[
\sum_{m_{1}\cdot m_{2}\cdots m_{n}=m}\prod_{i<j}\left(m_{i}-m_{j}\right)^{2}=\prod_{i=1}^{n}\prod_{j=i+1}^{n}(j-i)^{2}=\left(\prod_{i=1}^{n}(i-1)!\right)^{2}>0\]
so there is at least one term greater than zero in the sum in Corollary
\ref{cor:multiplicative function}. The other terms in the sum are
all greater or equal to zero so the sum is positive which proves the
lemma.\end{proof}
\begin{defn}
For each integer $n>0$ and integer $m>0$ we define the function
$h_{n}(m)$ by\[
h_{n}(m)=\frac{1}{n!}\sum_{m_{1}\cdot m_{2}\cdots m_{n}=m}\prod_{i<j}\left(m_{i}-m_{j}\right)^{2}.\]
For $n=2$ this function can be expressed in terms of arithmetic functions:\[
h_{2}(m)=\sigma_{2}(m)-md(m)\]
where $\sigma_{2}(m)$ is the divisor function $\sigma_{2}(m)=\sum_{d|m}d^{2}$
and $d(m)=\sum_{d|m}$ gives the number of divisors of $m$. We are
now in the position to state our second theorem. \end{defn}
\begin{thm}
For a Dirichlet series $F(s)$ with coefficients multiplicative coefficients
$f$ the Hankel determinant $H_{n}[F]$ is given by a Dirichlet series
with coefficients $h_{n}(m)f(m)$:\[
H_{n}^{(r)}[F]=\sum_{m=1}^{\infty}\frac{h_{n}(m)f(m)}{m^{2n+r}}.\]
\end{thm}
\begin{proof}
Inserting the definition of $h_{n}(m)$ in Corollary \ref{cor:multiplicative function}
gives the right hand expression.\end{proof}
\begin{example}
Our first example, already discussed in the introductions, is $f(n)=1$,
so that $F(s)=\zeta(s)$ where $\zeta(s)$ is the Riemann zeta function.
Using Theorem \ref{thm:HankelDeterminant} we obtain:\begin{equation}
H_{n}^{(r)}[\zeta]=\frac{1}{n!}\sum_{m_{1}\ldots m_{n}\ge1}\frac{1}{(m_{1}m_{2}\ldots m_{n})^{2n+r}}\prod_{i<j}\left(m_{i}-m_{j}\right)^{2}.\label{eq:H(zeta)}\end{equation}
In particular $H_{n}^{(r)}[\zeta]>0$ for any $n$. Specializing to
e.g. $n=2$ and $r=0$ we obtain from Theorem \ref{thm:HankelDeterminant}\begin{align*}
H_{2}[\zeta] & =\frac{1}{2}\sum_{m_{1},m_{2}\ge1}\frac{1}{m_{1}^{4}m_{2}^{4}}\left(m_{1}-m_{2}\right)^{2}\\
 & =\frac{1}{2}\sum_{m_{1},m_{2}\ge1}\left(\frac{1}{m_{1}^{4}m_{2}^{2}}-\frac{2}{m_{1}^{3}m_{2}^{3}}+\frac{1}{m_{1}^{2}m_{2}^{4}}\right)\\
 & =\zeta(2)\zeta(4)-\zeta(3)^{2}\\
 & =\left|\begin{array}{cc}
\zeta(2) & \zeta(3)\\
\zeta(3) & \zeta(4)\end{array}\right|\end{align*}
The Hankel determinant can also be expressed as a linear combination
of multiple zeta values\begin{eqnarray*}
H_{n}[\zeta] & = & \sum_{\pi\in\mathfrak{S_{n}}}\left(-1\right)^{\pi}\sum_{m_{n}>\ldots m_{2}>m_{1}\ge1}\frac{1}{m_{1}^{1+\pi(1)}m_{2}^{2+\pi(2)}\cdots m_{n}^{n+\pi(n)}}\\
 & = & \sum_{\pi\in\mathfrak{S_{n}}}\left(-1\right)^{\pi}\zeta(1+\pi(1),2+\pi(2),\ldots,n+\pi(n)).\end{eqnarray*}

\begin{example}
The second nontrivial example is\[
\det\left(\frac{1}{\zeta(i+j)}\right)_{1\le i,j\le n}=\sum_{m=1}^{\infty}\frac{\mu(m)}{m^{2n}}\left(\frac{1}{n!}\sum_{m_{1}\cdot m_{2}\cdots m_{n}=m}\prod_{i<j}\left(m_{i}-m_{j}\right)^{2}\right)\]
where $\mu(n)$ is the Möbius function.
\end{example}
\end{example}
\begin{rem*}
Multiple Dirichlet series have been investigated for some time for
a recent work see e.g. \cite{Goldfeld03} and \cite{MultipleZeta05}.
\end{rem*}

\section{A probabilistic model for the Asymptotic behavior of $H_{n}[\zeta]$}

In this section we determine the behavior of $H_{n}[\zeta]$ as $n\rightarrow\infty$.
The basic idea is to find the dominant contribution to the sum. We
note that all contributions are positive and that the Vandermonde
determinant is only nonzero if the $m_{i}$ are pairwise different.
We can reorder them so that $m_{n}>\ldots>m_{2}>m_{1}$. There are
precisely $n!$ contributions with $m_{1},m_{2}\ldots m_{n}$ in the
unordered sum. We write\begin{eqnarray}
H_{n}[\zeta] & = & \sum_{m_{n}>\ldots>m_{2}>m_{1}\ge1}\exp\left(\Phi(m_{1},m_{2}\ldots m_{n})\right)\label{eq:partition function}\end{eqnarray}
with\begin{equation}
\Phi(m_{1},m_{2}\ldots m_{n})=-2n\sum_{i=1}^{n}\log(m_{i})+2\sum_{i=1}^{n}\sum_{j=i+1}^{n}\log\left(m_{j}-m_{i}\right).\label{eq:Phi}\end{equation}

Finding the largest $\Phi(m_{1},m_{2}\ldots m_{n})$ is a discrete
combinatorial optimization problem. We can also view $\Phi$ as an
energy functional of a one-dimensional Coulomb gas problem on a lattice
where a large ($-2n$) attractive charge is placed at zero and the
charges placed at positions $m_{i}$ integer repel each other with
a logarithmic potential. Unlike in the standard one-dimensional Coulomb
gas problem, the charges cannot have a distance smaller than one.
Let us consider the case $n\gg1$. Building up the configuration by
adding a charge at $m_{i}$ one by one for the first few $m_{1},m_{2}\ldots$
the first term is dominant so that $\Phi$ is optimized by placing
the first charge at $m_{1}=1,$ the second at $m_{2}=2$ and so on.
Adding more charges slowly builds up the second term which produces
a repulsive potential which makes it more favorable to place charges
at $m_{i}>i$. From this analogy we expect the density of the $m_{i}$
to be one from $i=1$ up to some $i_{max}<n$ and then to decay to
zero as $i\rightarrow n$.

We define the distribution function for each configuration $\{m_{1},m_{2},\ldots m_{n}\}$
\begin{equation}
\rho(x)=\sum_{i=1}^{n}\delta(nx-m_{i})\label{eq:rho n}\end{equation}
with $\delta(x)$ being the delta distribution. The integral over
$\rho(x)$ is\begin{equation}
\int_{0}^{\infty}\rho(x)dx=1.\label{eq:normalization condition}\end{equation}
by definition. We call a distribution which obeys Eq. \ref{eq:normalization condition}
normalized. The functional $\Phi(m_{1},m_{2}\ldots m_{n})$ can be
expressed as a functional of the distribution function $\rho(x)$:\begin{eqnarray}\raisetag{-20pt}
\Phi[\rho] & = & -n^{2}\left(2\int_{0}^{\infty}\rho(x)\log(nx)dx-\iint_{0}^{\infty}\log|n(x-y)|\rho(x)\rho(y)dxdy\right)\nonumber \\
 & = & -n^{2}\left(2\log(n)+2\int_{0}^{\infty}\rho(x)\log(x)dx-\log(n)+\iint_{0}^{\infty}\rho(x)\rho(y)\log\left|x-y\right|dxdy\right)\nonumber \\
 & = & -n^{2}\log(n)-n^{2}\left(2\int_{0}^{\infty}\rho(x)\log(x)dx-\iint_{0}^{\infty}\rho(x)\rho(y)\log\left|x-y\right|dxdy\right).\nonumber\\&&\label{eq:functional}\end{eqnarray}

We seek a continuous test function $\rho(x)$ which maximizes the
functional Eq. \ref{eq:functional} subject to the constraint that
$0<\rho(x)\le1$ for all $x$ and $\rho(x)=0$ for $x<0$ and Eq.
\ref{eq:normalization condition}. We assume that such a continuous
functions exists. The arguments for its existence can probably be
made much more rigorous by using Poissonization, see e.g. \cite{Johansson98,Johansson01}
. The constraint Eq. \ref{eq:normalization condition} can be taken
care of by introducing a Lagrange multiplier $\lambda\in\mathbb{R}$
and finding an extremum of $\Phi[\rho]-\lambda\int_{0}^{\infty}\rho[x]dx$.
The problem of finding the extremum of $\Phi[\rho]$ is thus reduced
to the integral equation\begin{equation}
2\log(x)-2\int_{0}^{\infty}\rho(y)\log\left|x-y\right|=\lambda.\label{eq:minimization condition}\end{equation}
This integral equation only applies when $\rho(x)$ can actually be
varied, i.e. for $\rho(x)\ne1$ and $\rho(x)\ne0$. Differentiating
with respect to $x$ eliminates the Lagrange multiplier and finally
leads to\begin{equation}
\frac{1}{x}=\dashint_{0}^{\infty}\frac{\rho(y)}{x-y}dy\label{eq:integral equation}\end{equation}
where $\dashint$ denotes the principal value integral. Assume $\rho(x)<1$
for all $x\in\mathbb{R}$ then the integral equation has to be fulfilled
for all $x\in[0,\infty)$. The integral equation has as only solution
$\rho(x)=\delta(x)$ where $\delta(x)$ is the Dirac distribution
which diverges for $x\rightarrow0$ and does not fulfill the constraint.
Therefore the constraint $\rho(a)\le1$ has to be sharp for some positive
real number $a$. We will use an Ansatz for $\rho(x)$ and verify
that it obeys the integral equation and the constraint. 
\begin{thm}
\label{thm:maximizer}Define the function $\rho(x)$ for $x\in[0,\infty)$:
\[
\rho(x)=\begin{cases}
1 & (0\le x\le\frac{1}{2})\\
\frac{2}{\pi}\left(\arctan\left(\frac{1}{\sqrt{2x-1}}\right)-\frac{\sqrt{2x-1}}{2x}\right) & (x>\frac{1}{2})\end{cases}\]
Then $\rho(x)$ is the continuous solution of the stationary condition
Eq. \ref{eq:integral equation} for $x>1/2$, is normalized and fulfills
the condition $1\ge\rho(x)\ge0$ for all positive $x\in\mathbb{R}$.\end{thm}
\begin{proof}
The derivative of $\rho$ is:

\begin{equation}
\rho'(x)=\begin{cases}
0 & (x\le1/2)\\
-\frac{1}{\pi}\frac{1}{x^{2}\sqrt{2x-1}} & (x>1/2)\end{cases}.\label{eq:rho derivative}\end{equation}
 Let $\epsilon\in\mathbb{R}$ with $\epsilon>0$. Define $A_{\epsilon}(x)$
as \begin{equation}
A_{\epsilon}(x)=\frac{1}{2}\int_{0}^{\infty}\log\left|(x-y)^{2}+\epsilon^{2}\right|\rho'(y)dy.\label{eq:A}\end{equation}
where $\log(z)$ denotes the principal branch of the logarithm ($0\le\arg(z)\le\pi)$.
Using \ref{eq:rho derivative} and substituting $y=(s^{2}+1)/2$ we
have\begin{eqnarray*}
A_{\epsilon}(x) & = & -\frac{4}{\pi}\int_{0}^{\infty}\frac{\frac{1}{2}\log\left|\left(s^{2}+(1-2x)\right)+\epsilon^{2}\right|-\log(2)}{\left(s^{2}+1\right)^{2}}ds.\\
 & = & -\frac{1}{\pi}\int_{-\infty}^{+\infty}\frac{\log\left|\left(s^{2}+(1-2x)\right)^{2}+\epsilon^{2}\right|}{\left(s^{2}+1\right)^{2}}ds+\log(2)\end{eqnarray*}
First consider the case $x>1/2$. Using contour integration above
the real axis closing the contour in the upper half plane gives\begin{eqnarray*}
A_{\epsilon}(x) & = & \left.2\pi i\frac{d}{ds}\left(-\frac{1}{\pi}\frac{\log\left((s^{2}+1-2x)^{2}+\epsilon^{2}\right)}{(1+s)^{2}}\right)\right|_{s=i}+\log(2)\\
 & = & -2i\left[-\frac{4s(s^{2}+1-2x)}{\left(1+s^{2}-2x\right)^{2}+\epsilon^{2}}\frac{1}{(s+i)^{2}}-2\frac{\log(\left(1+s^{2}-2x\right)^{2}+\epsilon^{2})}{(s+i)^{3}}\right]_{s=i}+\log(2)\\
 & = & \frac{x}{x^{2}+\frac{\epsilon^{2}}{4}}-\log\left(\sqrt{x+\frac{\epsilon^{2}}{4}}\right).\end{eqnarray*}
Taking the limit $\epsilon\rightarrow0$ we have\[
\lim_{\epsilon\rightarrow0^{+}}A_{\epsilon}(x)=\frac{1}{x}-\log\left|x\right|.\]
Using integration by parts on Eq. \ref{eq:A} we have for all $\epsilon>0$
\[
A_{\epsilon}(x)=\frac{1}{2}\log\left(\left(x-y\right)^{2}+\epsilon^{2}\right)\rho(x)\left|_{x=1/2}^{x=\infty}\right.-\int_{1/2}^{\infty}\frac{\left(x-y\right)}{\left(x-y\right)^{2}+\epsilon^{2}}\rho(y)dy\]
In the limit $\epsilon\rightarrow0$ we have\[
\lim_{\epsilon\rightarrow0^{+}}A_{\epsilon}(x)=-\log(x)-\dashint_{0}^{\infty}\frac{\rho(y)}{x-y}dy\]
where $\dashint$ denotes the principal value integral. Combining
the two expressions for $A_{\epsilon}(x)$ we obtain \[
\frac{1}{x}=\dashint_{0}^{\infty}\frac{\rho(x)}{x-y}\]
for $x>1/2$ which is the first statement of the theorem. 

The normalization integral is $1$ since \begin{eqnarray*}
\int_{-\infty}^{+\infty}\rho(x)dx & = & \frac{1}{2}+\frac{2}{\pi}\int_{1/2}^{\infty}\left(\arctan\left(\frac{1}{\sqrt{2x-1}}\right)-\frac{\sqrt{2x-1}}{2x}\right)dx\\
 & = & \frac{1}{2}+\frac{4}{\pi}\int_{1/2}^{\infty}\frac{1}{y^{3}}\left(\arctan(y)-\frac{y}{1+y^{2}}\right)dy\\
 & = & \frac{1}{2}+\frac{1}{\pi}\left.\frac{\arctan(y)}{y^{2}}\right|_{y=1/2}^{y=\infty}-2\int_{1/2}^{\infty}\frac{1}{y(1+y^{2})}dy\\
 & = & 1.\end{eqnarray*}
The function $\rho(x)$ is continuous at $1/2$ since for $\epsilon\in\mathbb{R}$
with $\frac{1}{2}>\epsilon>0$\begin{eqnarray*}
\rho\left(\frac{1}{2}+\epsilon\right) & = & 1-O\left(\sqrt{\epsilon}\right).\end{eqnarray*}
and $\rho\left(\frac{1}{2}-\epsilon\right)=1$. The function $\rho(x)$
is continuous and monotonically decreasing with increasing $x$ since
$\rho'(x)<0$ for $x\in(1/2,\infty)$. The maximum of $\rho(x)$,
$x\in[0,\infty)$ is one and the infimum is $0$ since for $x\in[0,\infty)$
with $x\gg1$\[
\rho(x)=\frac{\sqrt{2}}{3\pi}x^{-3/2}+O\left(x^{-5/2}\right)>0.\]
which completes the proof. 
\end{proof}
To determine $\Phi[\rho]$ we need to evaluate the integral $\int_{0}^{\infty}\rho(y)\log\left|x-y\right|$.
For $x>1/2$ we can use the stationary condition Eq. \ref{eq:minimization condition}
however we have to determine $\lambda$ first. 
\begin{prop}
\label{pro:rho-integral}The function $\rho(x)$ $x\in(0,\infty)$
of Theorem \ref{thm:maximizer} has the property\[
\int_{0}^{\infty}\log\left|x-y\right|\rho(y)dy=\begin{cases}
\log(2)-1 & (x=0)\\
-\sqrt{1-2x}-\log(2)+2(x-1)\log\left(1+\sqrt{1-2x}\right)+x\log(2x) & (0<x<1/2)\\
\log(x) & (1/2\le x)\end{cases}\]
\end{prop}
\begin{proof}
Consider the case $x\ge1/2$ first. From theorem \ref{thm:maximizer}we
have \[
\int_{0}^{\infty}\rho(y)\log\left|x-y\right|dy=\log(x)+\frac{1}{2}\lambda.\]
To prove $\lambda=0$ it is sufficient to prove it for one $x\ge1/2$.
We choose $x=1/2$. Using the definition of $\rho(x)$ we have \begin{eqnarray*}
\int_{0}^{\infty}\rho(y)\log\left|\frac{1}{2}-y\right|dy & = & \int_{0}^{1/2}\log\left|\frac{1}{2}-y\right|dy+\int_{1/2}^{\infty}\left(\log\left(2y-1\right)-\log(2)\right)\rho(y)dy\\
 & = & \log\left(\frac{1}{2}\right)-\frac{1}{2}+\int_{1/2}^{\infty}\log(2y-1)\rho(y)dy\\
 & = & \log\left(\frac{1}{2}\right)-\frac{1}{2}-\int_{1/2}^{\infty}\left(\left(2y-1\right)\log\left(2y-1\right)-\left(2y-1\right)\right)\rho'(y)dy\\
 & = & \log\left(\frac{1}{2}\right)-\frac{1}{2}+\frac{2}{\pi}\int_{0}^{\infty}\frac{2s^{2}\log(s)-s^{2}}{\left(s^{2}+1\right)^{2}}ds\\
 & = & \log\left(\frac{1}{2}\right).\end{eqnarray*}
where we have substituted $s=1/\sqrt{2x-1}$ .The integral over $s$
can be done by contour integration. Finally we have\[
\int_{0}^{\infty}\rho(y)\log\left|\frac{1}{2}-y\right|=\log\left(\frac{1}{2}\right)=\log\left(\frac{1}{2}\right)+\frac{1}{2}\lambda.\]
This implies $\lambda=0$ which proves the case for $x>0$.

Next consider the case $0<x<1/2$ . We split the integral in two parts
and integrate by parts\begin{eqnarray*}
\int_{0}^{\infty}\rho(y)\log\left|x-y\right| & = & \int_{0}^{1/2}\log\left|x-y\right|dy+\int_{1/2}^{\infty}\rho(y)\log\left(y-x\right)dy\\
 & = & -x+x\log(x)+\int_{1/2}^{\infty}\rho'(y)\left(\left(y-x\right)\log\left(y-x\right)-(y-x)\right)dy\end{eqnarray*}
The second integral can be done by substituting $s=\sqrt{2y-1}$ .
The second term in the last line is\begin{eqnarray*}
\int_{1/2}^{\infty}\rho'(y)(y-x)\left(\log\left(y-x\right)-1\right)dy & = & \frac{2}{\pi}\int_{0}^{\infty}\frac{\log\left(s^{2}+1-2x\right)-\log(2)-1}{\left(s^{2}+1\right)^{2}}(s^{2}+1-2x)ds\\
 & = & \left(\log(2)+1\right)\left(x-1\right)+\int_{0}^{\infty}\frac{\log\left(s^{2}+1-2x\right)}{\left(s^{2}+1\right)^{2}}(s^{2}+1-2x)ds\\
 & = & \left(\log(2)+1\right)\left(x-1\right)+2\log\left(1+\sqrt{1-2x}\right)+\frac{x}{1+\sqrt{1-2x}}\end{eqnarray*}
where the last integral was evaluated using contour integration (see
e.g. \cite{GradsteinRyshik}, 4.295, integral 7). Collecting all terms
we obtain the $x<1/2$ case of the proposition. 

Finally consider the limit $x\rightarrow0$ with $x>0$. We have to
leading order in $x$\[
-\sqrt{1-2x}-\log(2)+2(x-1)\log\left(1+\sqrt{1-2x}\right)+x\log(2x)=-1+\log(2)+O(x\log(x)).\]
In the limit $x\rightarrow0$, $x>0$ we find\begin{eqnarray*}
\lim_{x\rightarrow0^{+}}\int_{0}^{\infty}\rho(y)\log\left|x-y\right|dy & = & -1+\log(2)\end{eqnarray*}
which completes the proof.
\end{proof}
We are now in the position to evaluate $\phi[\rho]$. Using Proposition
\ref{pro:rho-integral} all integrals of Eq. \ref{eq:partition function}
can be reduced to elementary integrals with the result\begin{eqnarray*}
\Phi\left[\rho\right] & = & -n^{2}\log(n)-n^{2}\left(2\int_{0}^{\infty}\rho(x)\log(x)dx-\iint_{0}^{\infty}\rho(x)\rho(y)\log\left|x-y\right|dxdy\right)\\
 & = & -n^{2}\log(n)-n^{2}\left(2\left(\log(2)-1\right)-\left(\log(2)-\frac{1}{2}\right)\right)\\
 & = & -n^{2}\left(\log(2n)-\frac{3}{2}\right).\end{eqnarray*}
From this we conclude that the dominant contribution to $H_{n}(\zeta)$
is\[
\log\left(H_{n}\left(\zeta\right)\right)\approx-n^{2}\left(\log(2n)-\frac{3}{2}\right)\]
which agrees with numerical findings.

\section{Relation to the Selberg integral}

We next discuss the relation of $H_{n}^{(r)}(\zeta)$ to the Selberg
integral (Selberg's extension of the beta integral \cite{Selberg44},
for a detailed explanation see \cite{Andrews99} chapter 8) which
plays a central role in random matrix theory (see \cite{Mehta04},
chapter 17) and is given by\[
S_{n}(\alpha,\beta,\gamma)=\int_{0}^{1}\cdots\int_{0}^{1}t_{1}^{\alpha-1}\cdots t_{n}^{\alpha-1}\left(1-t_{1}\right)^{\beta-1}\cdots\left(1-t_{n}\right)^{\beta-1}\prod_{i<j}\left(t_{i}-t_{j}\right)^{2\gamma}dt_{1}\cdots dt_{n}.\]
Substituting $t_{i}\rightarrow1/m_{i}$ we write\[
S_{n}(\alpha,\beta,\gamma)=\int_{1}^{\infty}\cdots\int_{1}^{\infty}\prod_{i=1}^{n}\frac{1}{m_{i}^{\alpha-1}}\left(1-\frac{1}{m_{i}}\right)^{\beta-1}\prod_{i<j}\left(\frac{1}{m_{i}}-\frac{1}{m_{j}}\right)^{2\gamma}dm_{1}\cdots dm_{n}.\]
For $\alpha=1+r$, $\beta=1$ and $\gamma=1$ we find\begin{eqnarray}
S_{n}\left(r+1,1\right) & = & \int_{1}^{\infty}\cdots\int_{1}^{\infty}\frac{1}{\left(m_{1}\cdots m_{n}\right)^{2n+r}}\prod_{i<j}\left(m_{i}-m_{j}\right)^{2}dm_{1}\cdots dm_{n}\label{eq:Selberg r}\end{eqnarray}
The similarity between Eq. \ref{eq:Selberg r} and Eq. \ref{eq:H(zeta)}
is striking and can be generalized easily. It can be seen immediately
from the definition that $S_{n}(1,1,1)>0$ for each $n>0$. 

However as we will prove now for $n\gg1$ the Selberg integral $S_{n}(1,1,1)$
is much larger than $H_{n}(\zeta)$ excluding a naive application
of Euler-MacLaurin summation formula to $H_{n}(\zeta)$. It is instructive
to repeat the saddle point analysis of the previous chapter for $S_{n}(1,1,1)$
in the limit $n\rightarrow\infty$. First note that \[
S_{n}(1,1,1)=\int_{1}^{\infty}\cdots\int_{1}^{\infty}\exp(\Phi[\rho_{\textrm{S}}])dm_{1}\cdots dm_{n}.\]
with the density \[
\rho_{\textrm{S}}(x)=\frac{1}{n}\sum_{i=1}^{n}\delta(x-m_{i}).\]
 Note that we did not rescale $x$. The density is normalized to one\[
\int_{-\infty}^{+\infty}\rho_{\textrm{S}}(x)dx=1.\]

\begin{prop}
The asymptotic density $\rho_{\textrm{S}}$ maximizing $\Phi[\rho]$
is\[
\rho_{\textrm{S}}(x)=\begin{cases}
0 & x\le1\\
\frac{1}{\pi}\frac{1}{x\sqrt{x-1}} & x>1\end{cases}.\]
\end{prop}
\begin{proof}
The condition that $\Phi[\rho]$ is stationary is\[
\frac{1}{x}=\dashint_{1}^{\infty}\frac{\rho_{\textrm{S}}(y)}{x-y}dy\]
the only difference to Eq. \ref{eq:minimization condition} is that
here no constraint $\rho_{\textrm{S}}(x)\le1$ has to be imposed since
there is no restriction for the difference $|m_{i}-m_{j}|$ for two
integration variables. The integral equation can again be solved by
standard methods \cite{Hochstadt73} and yields the result stated
above. Evaluating $\Phi[\rho_{\textrm{S}}]$ we find:\begin{eqnarray*}
\Phi[\rho_{\textrm{S}}] & = & -n^{2}\left(2\int\rho_{\textrm{S}}(x)\log(x)dx-\iint\rho_{\textrm{S}}(x)\rho_{\textrm{S}}(y)\log\left|x-y\right|dxdy\right)\\
 & = & -n^{2}\left(4\log(2)-2\log(2)\right)\\
 & = & -2\log(2)n^{2}\end{eqnarray*}
giving $S_{n}(1,1,1)\approx\exp(-2\log(2)n^{2}+\ldots)$ so that $S_{n}(1,1,1)\gg H_{n}(\zeta)>0$.
This is in fact the correct behavior as we will prove now.\end{proof}
\begin{prop}
The asymptotic behavior of $S_{n}(1,1,1)$ as $n\rightarrow\infty$
is\[
\log\left(S_{n}(1,1,1)\right)=-2\log\left(2\right)n^{2}+\left(\log(2\pi n)-1\right)n+O(1).\]
\end{prop}
\begin{proof}
The asymptotic behavior can be derived from the exact expression (Theorem
(8.1.1) in \cite{Andrews99})\begin{eqnarray*}
\log\left(S_{n}(1,1,1)\right) & = & \sum_{j=1}^{n}\left(2\log\left(\Gamma\left(j\right)\right)+\log(\Gamma(j+1))-\log\left(\Gamma(n+j)\right)\right)\\
 & = & 3\log\left(G(n+1)\right)+\log\left(G(n+2)\right)-\log\left(G(2n+1)\right).\end{eqnarray*}
where $\Gamma(x)$ is the Euler Gamma function and $G(n)=\prod_{i=0}^{n-2}i!$
for $n\in\mathbb{N}$, $n\ge0$ is the Barnes function. Inserting
the asymptotic expansion of the Barnes function \cite{Voros87} \[
\log(G(1+z))=z^{2}\left(\frac{1}{2}\log(z)-\frac{3}{4}\right)+\frac{1}{2}\log\left(2\pi\right)z-\frac{1}{12}\log\left(z\right)+O(1)\]
in the previous expression we find\[
\log(S_{n}(1,1,1))=-2\log(2)n^{2}+(\log(2\pi n)-1)n+O(1).\]
which completes the proof.
\end{proof}

\section{Relation to the Plancherel measure}

We finally remark on the relation of our results to the asymptotic
behavior of Plancherel measures. The Plancherel measure is defined
as\[
\mathcal{P}\left(\lambda\right)=\left(\frac{\dim\lambda}{\left|\lambda\right|!}\right)^{2}=\frac{\prod_{1\le i<j\le n}\left(m_{i}-m_{j}\right)^{2}}{\prod_{i=1}^{n}\left(m_{i}!\right)^{2}}\]
where $\dim\lambda$ is the dimension of the representation of the
symmetric group $\Sigma\left(\left|\lambda\right|\right)$ indexed
by the partition $\lambda$. The sum over all partitions is given
by\[
Z_{n}=\sum_{0\le m_{1}<\ldots<m_{n}}^{\infty}\frac{1}{(m_{1}!\cdots m_{n}!)^{2}}\prod_{i<j}\left(m_{i}-m_{j}\right)^{2}.\]
An analysis similar to the one above yields the following integral
equation for the density\[
\log(x)=\dashint_{0}^{\infty}\frac{\rho_{\textrm{P}}(y)}{x-y}dy.\]
In this case the solution has finite support and is given by\[
\rho_{\textrm{P}}(x)=\begin{cases}
\frac{1}{\pi}\arccos\left(1-x\right) & x\in[0,2]\\
0 & \textrm{otherwise}\end{cases}.\]
In this case the constraint $\rho_{\textrm{P}}(0)\le1$ does not restrict
the solution because $\rho_{\textrm{P}}(0)=1$ and $\rho_{\textrm{P}}(x)$
is monotonically decreasing for $x>0$. Below we show for comparison
the density $\rho(x)$ (solid line) and $\rho_{P}(x)$ (dashed line).

\selectlanguage{english}%
\begin{center}
\includegraphics{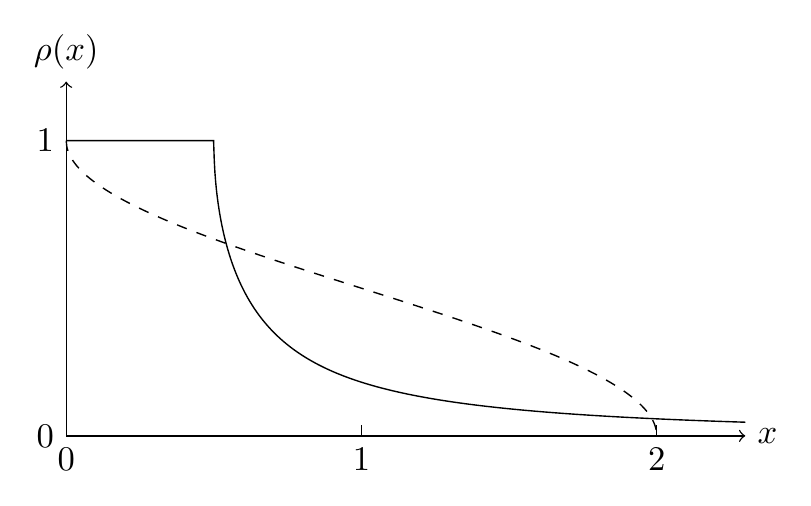}
\par\end{center}

\selectlanguage{american}%
\noindent Integrating $\rho_{P}(x)$ over $x$ we recover the famous
asymptotic behavior of Plancherel measure \cite{Vershik77,Logan77}.

\section{Acknowledgment}

\thanks{I would like to thank Don Zagier for his encouragement and comments,
Dorian Goldfeld for the idea to generalize Theorem \ref{thm:HankelDeterminant}
to multiplicative arithmetic functions, R. Flume, Piotr Su\l{}kowski
and A. Klemm for their constant interest as well as Kay Magaard and
A. M. Vershik for helpful discussions.}

\bibliographystyle{amsplain}

\begin{thebibliography}{10}

\bibitem{Andrews99}
E.~George Andrews, Richard Askey, and Roy Ranjan, \emph{Special {F}unctions},
  Encyclopedia of Mathematics and its Applications, vol.~71, Cambridge
  University Press, 1999.

\bibitem{Chapman96}
R.~Chapman, \emph{A {P}olynomial {T}aking {I}nteger {V}alues}, Math. Mag.
  \textbf{69} (1996), 121.

\bibitem{Goldfeld03}
A.~Diaconu, D.~Goldfeld, and J.~Hoffstein, \emph{Multiple {D}irichlet {S}eries
  and {M}oments of {Z}eta and {L}-functions}, Compositio Mathematica
  \textbf{64} (2003), no.~3, 297--360.

\bibitem{MultipleZeta05}
S.~Friedberg, Bump D., D.~Goldfeld, and J.~Hoffstein (eds.), \emph{Multiple
  {D}irichlet {S}eries, {A}utomorphic {F}orms, and {A}nalytic {N}umber
  {T}heory}, no.~75, American Mathematical Society, 2005.

\bibitem{GradsteinRyshik}
I.~S. Gradshteyn and I.~M. Ryzhik, \emph{Table of {I}ntegrals, {S}eries, and
  {P}roducts}, Academic Press, Inc., 1979.

\bibitem{Hochstadt73}
Harry Hochstadt, \emph{Integral {E}quations}, Pure and Applied Mathematics, Jon
  Wiley and {S}ons, Inc., 1973.

\bibitem{Johansson98}
Kurt Johansson, \emph{On {F}luctuations of {E}igenvalues of {R}andom
  {H}ermitian {M}atrices}, Duke Mathematical Journal \textbf{91} (1998), no.~1,
  151--203.

\bibitem{Johansson01}
\bysame, \emph{Discrete orthogonal polynomial ensembles and the {P}lancherel
  measure}, Annals of Mathematics \textbf{153} (2001), 259--296.

\bibitem{Logan77}
B.~F. Logan and L.~A. Shepp, \emph{A variational problem for random young
  tableaux}, Adv. Math. (1977), no.~26, 206--222.

\bibitem{Mehta04}
Madan~Lal Mehta, \emph{Random {M}atrices}, Pure and Applied Mathematics, vol.
  142, Elsevier, 2004.

\bibitem{Selberg44}
A.~Selberg, \emph{Bemerkninger om et multipelt integral}, Norske Mat. Tidsskr.
  \textbf{26} (1944), 71--78.

\bibitem{Vershik77}
A.~M. Vershik and S.~V. Kerov, \emph{Asymptotics of the {P}lancherel measure of
  he symmetric group and the limiting form of {Y}oung tableaux}, Dokl. Akad.
  Nauk SSSR \textbf{6} (1977), 1024--1027.

\bibitem{Voros87}
A.~Voros, \emph{Spectral functions, {S}pecial functions and the {S}elberg
  {Z}eta {F}unction}, Commun. Math. Phys. \textbf{110} (1987), 439--465.

\end{thebibliography}
\providecommand{\bysame}{\leavevmode\hbox to3em{\hrulefill}\thinspace}
\providecommand{\MR}{\relax\ifhmode\unskip\space\fi MR }
\providecommand{\MRhref}[2]{%
  \href{http://www.ams.org/mathscinet-getitem?mr=#1}{#2}
}
\providecommand{\href}[2]{#2}

\end{document}